\documentclass[a4paper,12pt,reqno]{amsart}
\usepackage{latexsym}
\usepackage{color}
\usepackage{amssymb,amsfonts,amsmath,mathrsfs}
\newcommand{\ze}[4]{\zeta^{#1}(1+\alpha_{#2}+\beta_{#3}+#4)}

\newtheorem{theorem}{Theorem}[section]

\newtheorem{lemma}{Lemma}[section]

\newtheorem{remark}{Remark}[section]

\begin{document}

\title[Amplified fourth moment of the Riemann zeta-function and applications]{Amplified fourth moment of the Riemann zeta-function and applications}
\author{Hung M. Bui, Richard R. Hall and Martin Subira Jorge}
\address{Department of Mathematics, University of Manchester, Manchester M13 9PL, UK}
\email{hung.bui@manchester.ac.uk}
\address{Department of Mathematics, University of York, York YO10 5DD, UK (deceased)}
\address{Girton College, University of Cambridge, Cambridge CB3 0JG, UK}
\email{ms2737@cam.ac.uk, msubirajorge@gmail.com}

\begin{abstract}
The twisted fourth moment of the Riemann zeta-function was established by Hughes and Young [J. Reine Angew. Math. 641 (2010), 203--236] and later improved by Bettin, Bui, Li and Radziwi\l\l\  [J. Eur. Math. Soc. (JEMS) 22 (2020), 3953--3980]. In applications one would often like to take the Dirichlet polynomial to mimic either $1/\zeta^r(s)$ (a mollifier) or $\zeta(s)^r$ (an amplifier) for some $r>0$. Previous known results include the mean value of the fourth power of $\zeta(s)$ times the square or the fourth power of a mollifier, or the square of an amplifier. In this paper we obtain the asymptotic formula for the fourth moment of the Riemann zeta-function times the fourth power of an amplifier. This has various
 applications to the theory of the Riemann zeta-function, e.g. gaps between zeros of $\zeta(s)$ and lower bounds for moments.
\end{abstract}

\allowdisplaybreaks

\maketitle

\section{Introduction and statement of results}

The distribution of the Riemann zeta-function on the critical line is a fascinating and challenging topic in number theory. A lot can be said about this by studying the moments of $\zeta(s)$. The $2k$-th moment of the Riemann zeta-function is defined as
\[
I_k(T)=\int_{0}^{T}|\zeta(\tfrac12+it)|^{2k}dt.
\]
Apart from the trivial case $k=0$, and the cases $k=1$, $k=2$ due to Hardy and Littlewood \cite{HL} and Ingham \cite{I}, respectively, no other asymptotic formulae is known for almost a century. Unconditional sharp upper bounds are known only for $0\leq k\leq 2$ \cite{H-B, BCR,HRS}. Note that the celebrated Lindel\"of Hypothesis, which states that $\zeta(1/2+it)\ll_\varepsilon(1+|t|)^\varepsilon$, is equivalent to the sharp bound $I_k(T)\ll_{\varepsilon,k}T^{1+\varepsilon}$ for all $k\in\mathbb{N}$.

To bridge the gap between Ingham's result for $k=2$ and the open problem for $k=3$, one can study the twisted moments of $\zeta(s)$. Let $$A(s)=\sum_{n\leq T^\vartheta}\frac{\alpha_n}{n^s}$$ be a Dirichlet polynomial of length $T^\vartheta$ with $\vartheta>0$ and $\alpha_n\ll n^\varepsilon$. Hughes and Young \cite{HY} established an asymptotic formula for the twisted fourth moment,
\begin{equation}\label{originaltwist}
\int_{0}^{T}|\zeta(\tfrac12+it)|^{4}|A(\tfrac12+it)|^2dt,
\end{equation}
for $\vartheta<1/11$. Their result was later improved to $\vartheta<1/4$ by Bettin, Bui, Li and Radziwi\l\l\ \cite{BBLR}. 

Depending on applications, one often takes $\alpha_n$ to approximate $\mu_r(n)$ or $d_r(n)$, for some $r>0$. The role of $A(s)$ there is to mimic $1/\zeta^r(s)$ or $\zeta^r(s)$, respectively. In these cases, $A(s)$ is called a ``mollifier" or an ``amplifier", respectively. When $\alpha_n\approx \mu_r(n)$, the asymptotic formulae for \eqref{originaltwist} and for 
\begin{equation}\label{modifiedtwist}
\int_{0}^{T}|\zeta(\tfrac12+it)|^{4}|A(\tfrac12+it)|^4dt
\end{equation}
were evaluated in \cite{B,BH}, and when $\alpha_n\approx d_r(n)$, the asymptotic formula for \eqref{originaltwist} was obtained in \cite{BM}. Our first objective in this paper is to establish the asymptotic formula for \eqref{modifiedtwist} when $\alpha_n\approx d_r(n)$.

Let $P(x)=\sum_{j}c_jx^j$ be a polynomial and let $A(s)$ be an amplifier of the form
\[
A(s):=A(s,P) = \sum_{n\leq y}\frac{d_r(n)P(\frac{\log y/n}{\log{y}})}{n^s},
\]
with $y=T^\vartheta$, $0<\vartheta<1$, and $r\in\mathbb{N}$\footnote{Our main result is for $r\in\mathbb{N}$ but can be easily modified to hold for all $r>0$.}. Let $w(t)$ be a smooth function with support in $[1,2]$ and satisfies $w^{(j)}(t)\ll_j T^\varepsilon$ for any $j\geq 0$. We consider the following slightly more general integral
\begin{align*}
I(\underline{\alpha},\underline{\beta})&=\int_{-\infty}^{\infty}\zeta(\tfrac{1}{2}+\alpha_1+it)\zeta(\tfrac{1}{2}+\alpha_2+it)\zeta(\tfrac{1}{2}+\alpha_3-it)\zeta(\tfrac{1}{2}+\alpha_4-it)\\
&\quad\times A(\tfrac{1}{2}+\beta_1+it)A(\tfrac{1}{2}+\beta_2+it)A(\tfrac{1}{2}+\beta_3-it)A(\tfrac{1}{2}+\beta_4-it)w\Big(\frac tT\Big)dt,
\end{align*}
where the shifts $\alpha_j,\beta_j\ll (\log T)^{-1}$.
Our main theorem is the following result.

\begin{theorem}\label{mainthm}
For any $\vartheta <1/8$ we have
\[
I(\underline{\alpha},\underline{\beta})=\frac{\widehat{\omega}(0)a_{2r+2}c(\underline{\alpha},\underline{\beta})}{2((r-1)!)^8((r^2-1)!)^4}T(\log y)^{4r^2+8r}(\log T)^{4}+O\big(T(\log T)^{4(r+1)^2-1}\big),
\]
where
\[
a_r=\prod_{p}\bigg(\Big(1-\frac{1}{p}\Big)^{r^2}\sum_{j=0}^{\infty}\frac{d_r(p^j)^2}{p^j}\bigg)
\]
and
\begin{align*}
	&c(\underline{\alpha},\underline{\beta})= \idotsint\limits_{\substack{0\leq v_j,t_j,x_j,z_j\leq 1\\t_1+t_2+x_3+z_3\leq1\\t_3+t_4+x_4+z_4\leq1\\t_1+t_3+x_1+z_1\leq 1\\t_2+t_4+x_2+z_2\leq 1 }}\big(1-\vartheta\sum x_j\big)\big(1-\vartheta\sum z_j\big)\\
&\quad\times\Big(v_1-v_2+\vartheta\big(x_1+x_2-z_1-z_2-v_1\sum x_j+v_2\sum z_j\big)\Big)\\
&\quad\times\Big(v_1-v_2+\vartheta\big(x_3+x_4-z_3-z_4-v_1\sum x_j+v_2\sum z_j\big)\Big)\\
&\quad\times y^{-\beta_1(t_1+t_2+x_3+z_3)-\beta_2(t_3+t_4+x_4+z_4)-\beta_3(t_1+t_3+x_1+z_1)-\beta_4(t_2+t_4+x_2+z_2)}\\
&\quad\times y^{-\alpha_1(x_1+x_2)-\alpha_2(z_1+z_2)-\alpha_3(x_3+x_4)-\alpha_4(z_3+z_4)}(Ty^{-\sum x_j})^{-(\alpha_1+\alpha_3)v_1}(Ty^{-\sum z_j})^{-(\alpha_2+\alpha_4)v_2}\\
&\quad\times \Big(T^{v_1-v_2}y^{x_1+x_2-z_1-z_2-v_1\sum x_j+v_2\sum z_j}\Big)^{(\alpha_1-\alpha_2)v_3}\\
&\quad\times\Big(T^{v_1-v_2}y^{x_3+x_4-z_3-z_4-v_1\sum x_j+v_2\sum z_j}\Big)^{(\alpha_3-\alpha_4)v_4}\\
&\quad\times (x_1x_2x_3x_4z_1z_2z_3z_4)^{r-1}(t_1t_2t_3t_4)^{r^2-1}P(1-t_1-t_2-x_3-z_3)\\
&\quad\times P(1-t_3-t_4-x_4-z_4)P(1-t_1-t_3-x_1-z_1)P(1-t_2-t_4-x_2-z_2)\\
&\quad\times dx_1dx_2dx_3dx_4dz_1dz_2dz_3dz_4dt_1dt_2dt_3dt_4dv_1dv_2dv_3dv_4
\end{align*}
uniformly for $\alpha_j,\beta_j\ll (\log T)^{-1}$.
\end{theorem}

Theorem \ref{mainthm} has various applications to the theory of the Riemann zeta-function, e.g. gaps between zeros of $\zeta(s)$ and lower bounds for moments. In fact, the latter has already been worked out in \cite{P1}, where Page obtained unconditional lower bounds for the sixth and eighth moments of $\zeta(s)$ utilizing our Theorem \ref{mainthm}. Here we shall demonstrate the former application, which concerns the vertical distribution of the zeros of the Riemann zeta-function on the critical line Re$(s)=1/2$. 

Let $t_n$ be the imaginary part of the $n$-th critical zero of the Riemann zeta-function of the form $\rho=1/2+it$ with $t>0$. All the following results in the paper apply to the sequence $\{t_n\}_{n\in\mathbb{N}}$ and are unconditional.

It is well-known that
\begin{equation}\label{NT}
N(T):=\sum_{0<\gamma\leq T}1=\frac{T\log T}{2\pi}-\frac{T}{2\pi}+O(\log T).
\end{equation}
We define
\begin{equation*}
\Lambda:=\limsup_{n\rightarrow\infty}\frac{t_{n+1}-t_n}{2\pi/\log t_n}.
\end{equation*}
It follows from the Riemann Hypothesis and \eqref{NT} that the average size of $\frac{t_{n+1}-t_n}{2\pi/\log t_n}$ is 1 as $n\rightarrow\infty$, and clearly \eqref{NT} implies that $\Lambda\geq1$. Using the pair correlation of the zeros of the Riemann zeta-function, Montgomery \cite{M} conjectured that $\Lambda=\infty$. That is to say that there exists arbitrarily large normalized gaps $\frac{t_{n+1}-t_n}{2\pi/\log t_n}$ between consecutive zeros on the critical line. The random matrix model even suggests that the normalized gaps  should get as large as $\sqrt{\log t_n}$ \cite[Section 1.3]{BAB}.

Selberg \cite[p.199]{S} remarked that he could prove $\Lambda > 1$, but did not provide any proof. All the other known lower bounds for $\Lambda$ were established using one of the two approaches, which were both due to the second named author and rely on some variants of the Wirtinger inequality and moments of the Riemann zeta-function. We shall describe them in details in Section \ref{sectionW}.

Hall showed that $\Lambda\geq\sqrt{11/2}=2.34\ldots$ in \cite{H2} and subsequently improved it to $\Lambda>2.63$ in \cite{H}. The former was achieved by using a variant of the Wirtinger inequality involving the $L^4$-norm of the Hardy $Z$-function  and its derivative (see Theorem \ref{Wi}), while the latter was proved by employing the standard $L^2$ form of the Wirtinger inequality applied to a product of two Hardy $Z$-functions with shifts. Hall's results were later improved to $\Lambda>2.76$ by Bredberg \cite{B1}, and the current best result is due to Bui and Milinovich \cite{BM}, where they obtained that $\Lambda>3.18$. Both these papers follow the second approach of Hall's. We shall show that Hall's first approach and our Theorem \ref{mainthm} lead to the following result.

\begin{theorem}\label{gapthm}
We have
\[
\Lambda>2.64.
\]
\end{theorem}

\begin{remark}
\emph{We have made no attempt to optimize the numerical bound in Theorem \ref{gapthm}. Our goal is to demonstrate the use of Theorem \ref{mainthm} in the study of gaps between zeros of the Riemann zeta-function. Theorem \ref{gapthm} is obtained with the simplest choice $r=1$ and $P(x)=1$. It is probable that with a higher degree polynomial $P(x)$ the bound for $\Lambda$ can be significantly improved.}
\end{remark}


\section{Auxiliary lemmas}

\begin{lemma}\label{6001}
Let $r\in\mathbb{N}$. Suppose $y_1\leq y_2$, $\alpha\ll(\log y_1)^{-1}$, and that $f$ and $g$ are smooth functions. Then we have
\begin{eqnarray*}
&&\sum_{n\leq y_1}\frac{d_r(n)}{n^{1+\alpha}}f\Big(\frac{\log y_1/n}{\log y_1}\Big)g\Big(\frac{\log y_2/n}{\log y_2}\Big)\\
&&\qquad=\frac{(\log y_1)^r}{(r-1)!}\int_{0}^{1}y_{1}^{-\alpha t}t^{r-1}f(1-t)g\Big(1-\frac{t\log y_1}{\log y_2}\Big)dt+O\big((\log y_1)^{r-1}\big).
\end{eqnarray*}
\end{lemma}
\begin{proof}
See \cite[Lemma 4.4]{BCY}.
\end{proof}

The next lemma is a generalization of \cite[Lemma 1]{BH}.

\begin{lemma}\label{EM}
Let $r\in\mathbb{N}$. Suppose $f_j$ are smooth functions for $1\leq j\leq 4$. Then we have
\begin{align*}
&\sum_{\substack{n_1n_2,n_3n_4\leq y\\n_1n_3,n_2n_4\leq y}}\frac{d_{r}(n_1)d_{r}(n_2)d_{r}(n_3)d_{r}(n_4)}{n_1^{1+\alpha_1}n_2^{1+\alpha_2}n_3^{1+\alpha_3}n_4^{1+\alpha_4}}\\
&\qquad\qquad\times f_1\Big(\frac{\log y/n_1n_2}{\log y}\Big)f_2\Big(\frac{\log y/n_3n_4}{\log y}\Big)f_3\Big(\frac{\log y/n_1n_3}{\log y}\Big)f_4\Big(\frac{\log y/n_2n_4}{\log y}\Big)\\
&\quad=\frac{(\log y)^{4r}}{((r-1)!)^4}\idotsint\limits_{\substack{0\leq t_j\leq 1\\t_1+t_2,t_3+t_4\leq1\\t_1+t_3,t_2+t_4\leq1 }}\ y^{-\alpha_1t_1-\alpha_2t_2-\alpha_3t_3-\alpha_4t_4}(t_1t_2t_3t_4)^{r-1}f_1(1-t_1-t_2)\\
&\qquad\qquad\times f_2(1-t_3-t_4)f_3(1-t_1-t_3)f_4(1-t_2-t_4)dt_1dt_2dt_3dt_4+O\big((\log y)^{4r-1}\big).
\end{align*}
\end{lemma}
\begin{proof}
We write
\begin{align*}
\sum_{\substack{n_1n_2,n_3n_4\leq y\\n_1n_3,n_2n_4\leq y}}&=\sum_{n_2\leq n_3\leq y}\sum_{n_1,n_4\leq y/n_3}+\sum_{n_3\leq n_2\leq y}\sum_{n_1,n_4\leq y/n_2}+\, O\big((\log y)^{2r}\big)\\
&=A_1+A_2+O\big((\log y)^{2r}\big),
\end{align*}
say.
By Lemma \ref{6001}, if $n_2\leq n_3$, then we have
\begin{align*}
&\sum_{\substack{n_1\leq y/n_3}}\frac{d_r(n_1)}{n_1^{1+\alpha_1}}f_1\Big(\frac{\log y/n_1n_2}{\log y}\Big)f_3\Big(\frac{\log y/n_1n_3}{\log y}\Big)=\frac{(\log \frac{y}{n_3})^r}{(r-1)!}\int_0^1\Big(\frac{y}{n_3}\Big)^{-\alpha_1t_1}t_1^{r-1}\\
&\qquad\times f_1\Big(\frac{(1-t_1)\log y/n_3}{\log y}+\frac{\log n_3/n_2}{\log y}\Big)f_3\Big(\frac{(1-t_1)\log y/n_3}{\log y}\Big)dt_1+O\big((\log y)^{r-1}\big).
\end{align*}
A similar expression holds for the sum over $n_4$ and hence
\begin{align*}
&A_1=\frac{1}{((r-1)!)^2}\int_0^1\int_0^1 y^{-\alpha_1t_1-\alpha_4t_4}(t_1t_4)^{r-1}\\
&\quad\times \sum_{n_3\leq y}\frac{d_r(n_3)}{n_3^{1+\alpha_3-\alpha_1t_1-\alpha_4t_4}}\Big(\log \frac{y}{n_3}\Big)^{2r}f_2\Big(\frac{(1-t_4)\log y/n_3}{\log y}\Big)f_3\Big(\frac{(1-t_1)\log y/n_3}{\log y}\Big)\\
&\quad\times \sum_{n_2\leq n_3}\frac{d_r(n_2)}{n_2^{1+\alpha_2}}f_1\Big(\frac{(1-t_1)\log y/n_3}{\log y}+\frac{\log n_3/n_2}{\log y}\Big)f_4\Big(\frac{(1-t_4)\log y/n_3}{\log y}+\frac{\log n_3/n_2}{\log y}\Big)dt_1dt_4\\
&\quad+O\big((\log y)^{4r-1}\big).
\end{align*}
An application of Lemma \ref{6001} to the sum over $n_2$ above, followed by another one to the sum over $n_3$ leads to
\begin{align*}
A_1&=\frac{(\log y)^{4r}}{((r-1)!)^4}\idotsint\limits_{\substack{0\leq t_j\leq 1}}y^{-\alpha_1t_1-\alpha_4t_4-(\alpha_3-\alpha_1t_1-\alpha_4t_4+\alpha_2t_2)t_3}(t_1t_2t_3t_4)^{r-1}t_3^r(1-t_3)^{2r}\\
&\qquad\times f_1\big((1-t_1)(1-t_3)+(1-t_2)t_3\big)f_2\big((1-t_4)(1-t_3)\big)f_3\big((1-t_1)(1-t_3)\big)\\
&\qquad\times f_4\big((1-t_4)(1-t_3)+(1-t_2)t_3\big)dt_1dt_2dt_3dt_4+O\big((\log y)^{4r-1}\big)\\
&=\frac{(\log y)^{4r}}{((r-1)!)^4}\int_0^1\int_0^{1-t_3}\int_0^{t_3}\int_0^{1-t_3}y^{-\alpha_1t_1-\alpha_2t_2-\alpha_3t_3-\alpha_4t_4}(t_1t_2t_3t_4)^{r-1}\\
&\qquad\times f_1(1-t_1-t_2)f_2(1-t_3-t_4)f_3(1-t_1-t_3)f_4\big(1-t_2-t_4)dt_1dt_2dt_4dt_3\\
&\qquad+O\big((\log y)^{4r-1}\big),
\end{align*}
after the changes of variables $t_1(1-t_3)\longrightarrow t_1$, $t_2t_3\longrightarrow t_2$ and $t_4(1-t_3)\longrightarrow t_4$. We obtain a similar expression for $A_2$ and the result hence follows.
\end{proof}

\begin{lemma}\label{600}
Let $r\in\mathbb{N}$ and let 
\begin{equation*}
K_j^n(\alpha,\beta)=\frac{1}{2\pi i}\int_{((\log T)^{-1})}\Big(\frac{y}{n}\Big)^u\zeta^r(1+\alpha+u)\zeta^r(1+\beta+u)\frac{du}{u^{j+1}}.
\end{equation*}
Then we have
\begin{eqnarray*}
K_j^n(\alpha,\beta)&=&\frac{(\log y/n)^{j+2r}}{((r-1)!)^2j!}\mathop{\int\int}_{\substack{0\leq x,z\leq 1\\x+z\leq 1}}\Big(\frac{y}{n}\Big)^{-\alpha x-\beta z}(xz)^{r-1}(1-x-z)^jdxdz\\
&&\qquad +O\big((\log T)^{j+2r-1}\big)
\end{eqnarray*}
uniformly for $\alpha,\beta\ll (\log T)^{-1}$.
\end{lemma}
\begin{proof}
See \cite[Lemma 4.1]{BM}.
\end{proof}

\section{Proof of Theorem \ref{mainthm}}

We shall need the following twisted fourth moment of the Riemann zeta-function \cite[Theorem 1.2]{BBLR}. 

\begin{theorem}[Bettin, Bui, Li and Radziwi\l\l]\label{BB}
Let $G(s)$ be an even entire function of rapid decay in any fixed strip $|\emph{Re}(s)|\leq C$ satisfying $G(0)=1$, and let
\begin{equation}\label{Vx}
V(x)=\frac{1}{2\pi i}\int_{(1)}G(s)(2\pi)^{-2s}x^{-s}\frac{ds}{s}.
\end{equation}
Then we have
\begin{align*}
&\sum_{m,n\leq x}\frac{a_m\overline{a_n}}{\sqrt{mn}}\int_{-\infty}^{\infty}\zeta(\tfrac{1}{2}+\alpha_1+it)\zeta(\tfrac{1}{2}+\alpha_2+it)\zeta(\tfrac{1}{2}+\alpha_3-it)\zeta(\tfrac{1}{2}+\alpha_4-it)\Big(\frac{m}{n}\Big)^{-it}w\Big(\frac{t}{T}\Big)dt\\
&\quad=\sum_{m,n\leq y}\frac{a_m\overline{a_n}}{\sqrt{mn}}\int_{-\infty}^{\infty}w\Big(\frac{t}{T}\Big)\bigg\{Z_{\alpha_1,\alpha_2,\alpha_3,\alpha_4,m,n}(t)+\Big(\frac{t}{2\pi}\Big)^{-(\alpha_1+\alpha_3)}Z_{-\alpha_3,\alpha_2,-\alpha_1,\alpha_4,m,n}(t)\\
&\quad\quad+\Big(\frac{t}{2\pi}\Big)^{-(\alpha_1+\alpha_4)}Z_{-\alpha_4,\alpha_2,\alpha_3,-\alpha_1,m,n}(t)+\Big(\frac{t}{2\pi}\Big)^{-(\alpha_2+\alpha_3)}Z_{\alpha_1,-\alpha_3,-\alpha_2,\alpha_4,m,n}(t)\\
&\quad\quad+\Big(\frac{t}{2\pi}\Big)^{-(\alpha_2+\alpha_4)}Z_{\alpha_1,-\alpha_4,\alpha_3,-\alpha_2,m,n}(t)+\Big(\frac{t}{2\pi}\Big)^{-(\alpha_1+\alpha_2+\alpha_3+\alpha_4)}Z_{-\alpha_3,-\alpha_4,-\alpha_1,-\alpha_2,m,n}(t)\bigg\}dt\\
&\quad\quad+O_\varepsilon(T^{1/2+\varepsilon}x^2+T^{3/4+\varepsilon}x)
\end{align*}
uniformly for $\alpha_j\ll (\log T)^{-1}$, where 
\[
Z_{\alpha_1,\alpha_2,\alpha_3,\alpha_4,m,n}(t)=\sum_{ma_1a_2=na_3a_4}\frac{1}{a_{1}^{1/2+\alpha_1}a_{2}^{1/2+\alpha_2}a_{3}^{1/2+\alpha_3}a_{4}^{1/2+\alpha_4}}V\Big(\frac{a_1a_2a_3a_4}{t^2}\Big).
\]
\end{theorem}

Using Theorem \ref{BB} we can write 
\begin{align*}
I(\underline{\alpha},\underline{\beta})=I_1+I_2+I_3+I_4+I_5+I_6+O_\varepsilon(T^{1/2+4\vartheta+\varepsilon}+T^{3/4+2\vartheta+\varepsilon}).
\end{align*}
In order to simplify the calculations later, we shall choose $G(s)$ to have zeros at $s=-(\alpha_j+\alpha_k)/2$ for  $1\leq j\leq2$ and $3\leq k\leq4$. 

We first consider $I_1$, which is
\begin{align*}
I_1=&\sum_{m_1,m_2,m_3,m_4\leq y} \frac{d_r(m_1)d_r(m_2)d_r(m_3)d_r(m_4)P(\frac{\log y/m_1}{\log{y}})P(\frac{\log y/m_2}{\log{y}})P(\frac{\log y/m_3}{\log{y}})P(\frac{\log y/m_4}{\log{y}})}{m_1^{1/2+\beta_1}m_2^{1/2+\beta_2}m_3^{1/2+\beta_3}m_4^{1/2+\beta_4}}\nonumber\\
&\qquad\times\sum_{\substack{m_1m_2a_1a_2\\=m_3m_4a_3a_4}}\frac{1}{a_1^{1/2+\alpha_1}a_2^{1/2+\alpha_2}a_3^{1/2+\alpha_3}a_4^{1/2+\alpha_4}}\int_{-\infty}^{\infty}\omega\Big(\frac{t}{T}\Big)V\left(\frac{a_1a_2a_3a_4}{t^2}\right)\,dt.
\end{align*}
Note that
\[
P\Big(\frac{\log y/n}{\log{y}}\Big)=\sum_{j}\frac{c_jj!}{(\log y)^j}\frac{1}{2\pi i}\int_{(1)}\Big(\frac yn\Big)^u\frac{du}{u^{j+1}}.
\]
Together with \eqref{Vx} we then have that
\begin{align}\label{I1}
I_1=&\sum_{j_1,j_2,j_3,j_4}\frac{c_{j_1}c_{j_2}c_{j_3}c_{j_4}j_1!j_2!j_3!j_4!}{(\log y)^{j_1+j_2+j_3+j_4}}\Big(\frac{1}{2\pi i}\Big)^5\int_{-\infty}^{\infty} \idotsint_{(1)^5}\omega\Big(\frac{t}{T}\Big)G(s)\Big(\frac{t}{2\pi}\Big)^{2s}y^{u_1+u_2+u_3+u_4}\nonumber\\
&\qquad\times\sum_{\substack{m_1m_2a_1a_2\\=m_3m_4a_3a_4}}\frac{d_r(m_1)d_r(m_2)d_r(m_3)d_r(m_4)}{m_1^{1/2+\beta_1+u_1}m_2^{1/2+\beta_2+u_2}m_3^{1/2+\beta_3+u_3}m_4^{1/2+\beta_4+u_4}}\\
&\qquad\qquad\times\frac{1}{a_1^{1/2+\alpha_1+s}a_2^{1/2+\alpha_2+s}a_3^{1/2+\alpha_3+s}a_4^{1/2+\alpha_4+s}}\frac{du_1}{u_1^{j_1+1}}\frac{du_2}{u_2^{j_2+1}}\frac{du_3}{u_3^{j_3+1}}\frac{du_4}{u_4^{j_4+1}}\frac{ds}{s}dt.\nonumber
\end{align} Using the Euler product formula, the innermost sum is equal to
\begin{align}\label{I355}
&A(\underline{\alpha},\underline{\beta},\underline{u},s)\zeta(1+\alpha_1+\alpha_3+2s)\zeta(1+\alpha_1+\alpha_4+2s)\zeta(1+\alpha_2+\alpha_3+2s)\nonumber\\
&\quad\times\zeta(1+\alpha_2+\alpha_4+2s) \zeta^{r^2}(1+\beta_1+\beta_3+u_1+u_3)\zeta^{r^2}(1+\beta_1+\beta_4+u_1+u_4)\nonumber\\
    &\quad\times\zeta^{r^2}(1+\beta_2+\beta_3+u_2+u_3)\zeta^{r^2}(1+\beta_2+\beta_4+u_2+u_4)\ze{r}{1}{3}{u_3+s}\nonumber\\
    &\quad\times\ze{r}{1}{4}{u_4+s}\ze{r}{2}{3}{u_3+s}\ze{r}{2}{4}{u_4+s}\\
    &\quad\times\ze{r}{3}{1}{u_1+s}\ze{r}{3}{2}{u_2+s}  \ze{r}{4}{1}{u_1+s}\nonumber\\
    &\quad\times \ze{r}{4}{2}{u_2+s},\nonumber
\end{align}
where $A(\underline{\alpha},\underline{\beta},\underline{u},s)$ is an Euler product converging absolutely in a product of half-planes containing the origin.

We first move the $u_j$ contours, $1\leq j\leq 4$, in \eqref{I1} to $\textrm{Re}(u_j)=\delta$, and then move the $s$ contour to $\textrm{Re}(s)=-\delta/2$, where
$\delta> 0$ is some fixed constant such that the arithmetical factor converges absolutely. In doing so, we only cross a simple pole at $s=0$. Note that the poles at $s=-(\alpha_j+\alpha_k)/2$, for every $1\leq j\leq2$ and $3\leq k\leq4$, of the zeta-functions are cancelled out by the zeros of $G(s)$. On the new line we simply bound the integral by absolute values, giving a contribution
\[
\ll_\varepsilon T^{1-\delta+\varepsilon}y^{4\delta}\ll_\varepsilon T^{1-\varepsilon}.
\]
Hence up to an error of size $O_\varepsilon (T^{1-\varepsilon})$ we have
\begin{align*}
&I_1=\widehat{\omega}(0)T\zeta(1+\alpha_1+\alpha_3)\zeta(1+\alpha_1+\alpha_4)\zeta(1+\alpha_2+\alpha_3)\zeta(1+\alpha_2+\alpha_4)\\
&\quad\times\sum_{j_1,j_2,j_3,j_4}\frac{c_{j_1}c_{j_2}c_{j_3}c_{j_4}j_1!j_2!j_3!j_4!}{(\log y)^{j_1+j_2+j_3+j_4}}\Big(\frac{1}{2\pi i}\Big)^4\iiiint_{(\delta)^4}y^{u_1+u_2+u_3+u_4}A(\underline{\alpha},\underline{\beta},\underline{u},0)\\
&\quad\times \zeta^{r^2}(1+\beta_1+\beta_3+u_1+u_3)\zeta^{r^2}(1+\beta_1+\beta_4+u_1+u_4)\zeta^{r^2}(1+\beta_2+\beta_3+u_2+u_3)\\
    &\quad\times\zeta^{r^2}(1+\beta_2+\beta_4+u_2+u_4)\ze{r}{1}{3}{u_3}\ze{r}{1}{4}{u_4}\\
    &\quad\times\ze{r}{2}{3}{u_3}\ze{r}{2}{4}{u_4}\ze{r}{3}{1}{u_1}\\
    &\quad\times\ze{r}{3}{2}{u_2}  \ze{r}{4}{1}{u_1}\ze{r}{4}{2}{u_2}\frac{du_1}{u_1^{j_1+1}}\frac{du_2}{u_2^{j_2+1}}\frac{du_3}{u_3^{j_3+1}}\frac{du_4}{u_4^{j_4+1}}.\nonumber
\end{align*}

Next we move the contours of integration to $\textrm{Re}(u_j)\asymp (\log T)^{-1}$, $1\leq j\leq 4$. Bounding the integrals trivially shows that $I_1\ll T(\log T)^{4(r+1)^2}$. Hence we can replace $A(\underline{\alpha},\underline{\beta},\underline{u},0)$ by $A(\underline{0})$ with an error of size $O(T(\log T)^{4(r+1)^2-1})$. By letting $\alpha_j=\beta_j=0$ and $u_j=s$, $1\leq j\leq 4$, in \eqref{I355} we get
\begin{align*}
A(\underline{0},s,s,s,s,s)&=\zeta(1+2s)^{-4(r+1)^2}\sum_{{m_1m_2a_1a_2=m_3m_4a_3a_4}}\frac{d_r(m_1)d_r(m_2)d_r(m_3)d_r(m_4)}{(m_1m_2m_3m_4a_1a_2a_3a_4)^{1/2+s}}\\
&=\zeta(1+2s)^{-4(r+1)^2}\sum_{n=1}^{\infty}\frac{d_{2r+2}(n)^2}{n^{1+2s}}\\
&=\prod_{p}\bigg(\Big(1-\frac{1}{p^{1+2s}}\Big)^{4(r+1)^2}\sum_{j=0}^\infty\frac{d_{2r+2}(p^j)^2}{p^{j(1+2s)}}\bigg),
\end{align*}
and so $A(\underline{0})=a_{2r+2}$, in particular. Thus
\begin{align}\label{I1f}
I_1&=\widehat{\omega}(0)a_{2r+2}T\zeta(1+\alpha_1+\alpha_3)\zeta(1+\alpha_1+\alpha_4)\zeta(1+\alpha_2+\alpha_3)\zeta(1+\alpha_2+\alpha_4)\nonumber\\
&\qquad\times\sum_{j_1,j_2,j_3,j_4}\frac{c_{j_1}c_{j_2}c_{j_3}c_{j_4}j_1!j_2!j_3!j_4!}{(\log y)^{j_1+j_2+j_3+j_4}}\,J_{j_1,j_2,j_3,j_4}+O\big(T(\log T)^{4(r+1)^2-1}\big),
\end{align}
where $J_{j_1,j_2,j_3,j_4}$ is given by
\begin{align*}
&\Big(\frac{1}{2\pi i}\Big)^4\iiiint_{( (\log T)^{-1})^4}y^{u_1+u_2+u_3+u_4}\zeta^{r^2}(1+\beta_1+\beta_3+u_1+u_3)\zeta^{r^2}(1+\beta_1+\beta_4+u_1+u_4)\\
    &\quad\times\zeta^{r^2}(1+\beta_2+\beta_3+u_2+u_3)\zeta^{r^2}(1+\beta_2+\beta_4+u_2+u_4)\ze{r}{1}{3}{u_3}\\
    &\quad\times\ze{r}{1}{4}{u_4}\ze{r}{2}{3}{u_3}\ze{r}{2}{4}{u_4}\ze{r}{3}{1}{u_1}\\
    &\quad\times\ze{r}{3}{2}{u_2}  \ze{r}{4}{1}{u_1}\ze{r}{4}{2}{u_2}\frac{du_1}{u_1^{j_1+1}}\frac{du_2}{u_2^{j_2+1}}\frac{du_3}{u_3^{j_3+1}}\frac{du_4}{u_4^{j_4+1}}.
\end{align*}

We express the $\zeta^{r^2}$ terms as absolutely convergent Dirichlet series, getting
\begin{align*}
J_{j_1,j_2,j_3,j_4}=&\sum_{\substack{n_1n_2,n_3n_4\leq y\\n_1n_3,n_2n_4\leq y}}\frac{d_{r^2}(n_1)d_{r^2}(n_2)d_{r^2}(n_3)d_{r^2}(n_4)}{n_1^{1+\beta_1+\beta_3}n_2^{1+\beta_1+\beta_4}n_3^{1+\beta_2+\beta_3}n_4^{1+\beta_2+\beta_4}}\\
&\quad\times\frac{1}{2\pi i}\int_{( (\log T)^{-1})}\Big(\frac{y}{n_1n_3}\Big)^{u_3}\zeta ^r(1+\alpha_1+\beta_3+u_3)\zeta ^r(1+\alpha_2+\beta_3+u_3)\frac{du_3}{u_3^{j_3+1}}\\
&\quad\times\frac{1}{2\pi i}\int_{( (\log T)^{-1})}\Big(\frac{y}{n_2n_4}\Big)^{u_4}\zeta ^r(1+\alpha_1+\beta_4+u_4)\zeta ^r(1+\alpha_2+\beta_4+u_4)\frac{du_4}{u_4^{j_4+1}}\\
&\quad\times\frac{1}{2\pi i}\int_{( (\log T)^{-1})}\Big(\frac{y}{n_1n_2}\Big)^{u_1}\zeta ^r(1+\alpha_3+\beta_1+u_1)\zeta ^r(1+\alpha_4+\beta_1+u_1)\frac{du_1}{u_1^{j_1+1}}\\
&\quad\times\frac{1}{2\pi i}\int_{( (\log T)^{-1})}\Big(\frac{y}{n_3n_4}\Big)^{u_2}\zeta ^r(1+\alpha_3+\beta_2+u_2)\zeta ^r(1+\alpha_4+\beta_2+u_2)\frac{du_2}{u_2^{j_2+1}}\\
=&\sum_{n_1,n_2,n_3,n_4\leq y}\frac{d_{r^2}(n_1)d_{r^2}(n_2)d_{r^2}(n_3)d_{r^2}(n_4)}{n_1^{1+\beta_1+\beta_3}n_2^{1+\beta_1+\beta_4}n_3^{1+\beta_2+\beta_3}n_4^{1+\beta_2+\beta_4}}K_{j_3}^{n_1n_3}(\alpha_1+\beta_3,\alpha_2+\beta_3)\\
&\quad\times K_{j_4}^{n_2n_4}(\alpha_1+\beta_4,\alpha_2+\beta_4) K_{j_1}^{n_1n_2}(\alpha_3+\beta_1,\alpha_4+\beta_1)K_{j_2}^{n_3n_4}(\alpha_3+\beta_2,\alpha_4+\beta_2),
\end{align*}
where $K_j^n(\alpha,\beta)$  is defined in Lemma \ref{600}. Note that here we are able to restrict the sum over $n_j$ to $n_1n_2,n_3n_4,n_1n_3,n_2n_4\leq y$ by moving the line integrals far to the right. From Lemma \ref{600} we have
\begin{align*}
&J_{j_1,j_2,j_3,j_4}=\frac{(\log y)^{j_1+j_2+j_3+j_4}}{((r-1)!)^8j_1!j_2!j_3!j_4!}\idotsint\limits_{\substack{0\leq x_j,z_j\leq 1\\x_1+z_1,x_2+z_2\leq 1\\x_3+z_3,x_4+z_4\leq1}}\\
&\quad\times (x_1x_2x_3x_4z_1z_2z_3z_4)^{r-1}(1-x_1-z_1)^{j_3}(1-x_2-z_2)^{j_4}(1-x_3-z_3)^{j_1}(1-x_4-z_4)^{j_2}\\
&\quad\times\sum_{\substack{n_1n_2,n_3n_4\leq y\\n_1n_3,n_2n_4\leq y}}\frac{d_{r^2}(n_1)d_{r^2}(n_2)d_{r^2}(n_3)d_{r^2}(n_4)}{n_1^{1+\beta_1+\beta_3}n_2^{1+\beta_1+\beta_4}n_3^{1+\beta_2+\beta_3}n_4^{1+\beta_2+\beta_4}}\\
&\quad\times\Big(\frac{\log y/n_1n_2}{\log y}\Big)^{j_1+2r}\Big(\frac{\log y/n_3n_4}{\log y}\Big)^{j_2+2r}\Big(\frac{\log y/n_1n_3}{\log y}\Big)^{j_3+2r}\Big(\frac{\log y/n_2n_4}{\log y}\Big)^{j_4+2r}\\
&\quad\times\Big(\frac{y}{n_1n_2}\Big)^{-(\alpha_3+\beta_1)x_3-(\alpha_4+\beta_1)z_3}\Big(\frac{y}{n_3n_4}\Big)^{-(\alpha_3+\beta_2)x_4-(\alpha_4+\beta_2)z_4}\Big(\frac{y}{n_1n_3}\Big)^{-(\alpha_1+\beta_3)x_1-(\alpha_2+\beta_3)z_1}\\
&\quad\times\Big(\frac{y}{n_2n_4}\Big)^{-(\alpha_1+\beta_4)x_2-(\alpha_2+\beta_4)z_2}
dx_1dx_2dx_3dx_4dz_1dz_2dz_3dz_4.
\end{align*}
Using Lemma \ref{EM} we deduce that
\begin{align*}
&J_{j_1,j_2,j_3,j_4}=\frac{(\log y)^{j_1+j_2+j_3+j_4+4r^2+8r}}{((r-1)!)^8((r^2-1)!)^4j_1!j_2!j_3!j_4!}\ \idotsint\limits_{\substack{0\leq t_j\leq 1\\t_1+t_2,t_3+t_4\leq1\\t_1+t_3,t_2+t_4\leq1 }}\quad\idotsint\limits_{\substack{0\leq x_j,z_j\leq 1\\x_1+z_1,x_2+z_2\leq 1\\x_3+z_3,x_4+z_4\leq1}}\\
&\quad\times y^{-(\alpha_1+\beta_3)x_1-(\alpha_2+\beta_3)z_1-(\alpha_1+\beta_4)x_2-(\alpha_2+\beta_4)z_2-(\alpha_3+\beta_1)x_3-(\alpha_4+\beta_1)z_3-(\alpha_3+\beta_2)x_4-(\alpha_4+\beta_2)z_4}\\
&\quad\times y^{-((\beta_1+\beta_3)-(\alpha_1+\beta_3)x_1-(\alpha_2+\beta_3)z_1-(\alpha_3+\beta_1)x_3-(\alpha_4+\beta_1)z_3)t_1}\\
&\quad\times y^{-((\beta_1+\beta_4)-(\alpha_1+\beta_4)x_2-(\alpha_2+\beta_4)z_2-(\alpha_3+\beta_1)x_3-(\alpha_4+\beta_1)z_3)t_2}\\
&\quad\times y^{-((\beta_2+\beta_3)-(\alpha_1+\beta_3)x_1-(\alpha_2+\beta_3)z_1-(\alpha_3+\beta_2)x_4-(\alpha_4+\beta_2)z_4)t_3}\\
&\quad\times y^{-((\beta_2+\beta_4)-(\alpha_1+\beta_4)x_2-(\alpha_2+\beta_4)z_2-(\alpha_3+\beta_2)x_4-(\alpha_4+\beta_2)z_4)t_4}\\
&\quad\times (x_1x_2x_3x_4z_1z_2z_3z_4)^{r-1}(1-x_1-z_1)^{j_3}(1-x_2-z_2)^{j_4}(1-x_3-z_3)^{j_1}(1-x_4-z_4)^{j_2}\\
&\quad\times (t_1t_2t_3t_4)^{r^2-1}(1-t_1-t_2)^{j_1+2r}(1-t_3-t_4)^{j_2+2r}(1-t_1-t_3)^{j_3+2r}(1-t_2-t_4)^{j_4+2r}\\
&\quad\times dx_1dx_2dx_3dx_4dz_1dz_2dz_3dz_4dt_1dt_2dt_3dt_4.
\end{align*}
With the change of variables $x_1(1-t_1-t_3)\rightarrow x_1$, $z_1(1-t_1-t_3)\rightarrow z_1$, $x_2(1-t_2-t_4)\rightarrow x_2$, $z_2(1-t_2-t_4)\rightarrow z_2$, $x_3(1-t_1-t_2)\rightarrow x_3$, $z_3(1-t_1-t_2)\rightarrow z_3$, $x_4(1-t_3-t_4)\rightarrow x_4$, $z_4(1-t_3-t_4)\rightarrow z_4$  we obtain that
\begin{align*}
&J_{j_1,j_2,j_3,j_4}=\frac{(\log y)^{j_1+j_2+j_3+j_4+4r^2+8r}}{((r-1)!)^8((r^2-1)!)^4j_1!j_2!j_3!j_4!}\ \idotsint\limits_{\substack{0\leq t_j,x_j,z_j\leq 1\\t_1+t_2+x_3+z_3\leq1\\t_3+t_4+x_4+z_4\leq1\\t_1+t_3+x_1+z_1\leq 1\\t_2+t_4+x_2+z_2\leq 1 }}\\
&\qquad\times y^{-(\beta_1+\beta_3)t_1-(\beta_1+\beta_4)t_2-(\beta_2+\beta_3)t_3-(\beta_2+\beta_4)t_4}\\
&\qquad\times y^{-(\alpha_1+\beta_3)x_1-(\alpha_2+\beta_3)z_1-(\alpha_1+\beta_4)x_2-(\alpha_2+\beta_4)z_2-(\alpha_3+\beta_1)x_3-(\alpha_4+\beta_1)z_3-(\alpha_3+\beta_2)x_4-(\alpha_4+\beta_2)z_4}\\
&\qquad\times (x_1x_2x_3x_4z_1z_2z_3z_4)^{r-1}(t_1t_2t_3t_4)^{r^2-1}(1-t_1-t_2-x_3-z_3)^{j_1}\\
&\qquad\times (1-t_3-t_4-x_4-z_4)^{j_2}(1-t_1-t_3-x_1-z_1)^{j_3}(1-t_2-t_4-x_2-z_2)^{j_4}\\
&\qquad\times dx_1dx_2dx_3dx_4dz_1dz_2dz_3dz_4dt_1dt_2dt_3dt_4.
\end{align*}
So from \eqref{I1f} we get
\begin{align*}
&I_1=\frac{\hat{\omega}(0)a_{2r+2}T(\log y)^{4r^2+8r}}{((r-1)!)^8((r^2-1)!)^4}\zeta(1+\alpha_1+\alpha_3)\zeta(1+\alpha_1+\alpha_4)\zeta(1+\alpha_2+\alpha_3)\zeta(1+\alpha_2+\alpha_4)\\
&\qquad\times\ \idotsint\limits_{\substack{0\leq t_j,x_j,z_j\leq 1\\t_1+t_2+x_3+z_3\leq1\\t_3+t_4+x_4+z_4\leq1\\t_1+t_3+x_1+z_1\leq 1\\t_2+t_4+x_2+z_2\leq 1 }}y^{-\alpha_1(x_1+x_2)-\alpha_2(z_1+z_2)-\alpha_3(x_3+x_4)-\alpha_4(z_3+z_4)}\\
&\qquad\times y^{-\beta_1(t_1+t_2+x_3+z_3)-\beta_2(t_3+t_4+x_4+z_4)-\beta_3(t_1+t_3+x_1+z_1)-\beta_4(t_2+t_4+x_2+z_2)}\\
&\qquad\times(x_1x_2x_3x_4z_1z_2z_3z_4)^{r-1}(t_1t_2t_3t_4)^{r^2-1}P(1-t_1-t_2-x_3-z_3)\\
&\qquad\times P(1-t_3-t_4-x_4-z_4)P(1-t_1-t_3-x_1-z_1)P(1-t_2-t_4-x_2-z_2)\\
&\qquad\times dx_1dx_2dx_3dx_4dz_1dz_2dz_3dz_4dt_1dt_2dt_3dt_4+O\big(T(\log T)^{4(r+1)^2-1}\big).
\end{align*}

Note that $I_2$ is obtained by multiplying $I_1$ by $T^{-(\alpha_1+\alpha_3)}$ and changing the shifts $\alpha_1\longleftrightarrow-\alpha_3$, $I_3$ is obtained by multiplying $I_1$ by $T^{-(\alpha_1+\alpha_4)}$ and changing the shifts $\alpha_1\longleftrightarrow-\alpha_4$, $I_4$ is obtained by multiplying $I_1$ by $T^{-(\alpha_2+\alpha_3)}$ and changing the shifts $\alpha_2\longleftrightarrow-\alpha_3$, $I_5$ is obtained by multiplying $I_1$ by $T^{-(\alpha_2+\alpha_4)}$ and changing the shifts $\alpha_2\longleftrightarrow-\alpha_4$, and $I_6$ is obtained by multiplying $I_1$ by $T^{-(\alpha_1+\alpha_2+\alpha_3+\alpha_4)}$ and changing the shifts $\alpha_1\longleftrightarrow-\alpha_3$ and $\alpha_2\longleftrightarrow-\alpha_4$. Hence
\begin{align}\label{Uformula}
&I(\underline{\alpha},\underline{\beta})=\frac{\hat{\omega}(0)a_{2r+2}T(\log y)^{4r^2+8r}}{((r-1)!)^8((r^2-1)!)^4}\ \idotsint\limits_{\substack{0\leq t_j,x_j,z_j\leq 1\\t_1+t_2+x_3+z_3\leq1\\t_3+t_4+x_4+z_4\leq1\\t_1+t_3+x_1+z_1\leq 1\\t_2+t_4+x_2+z_2\leq 1 }}U(\underline{x},\underline{z})\nonumber\\
&\quad\times y^{-\beta_1(t_1+t_2+x_3+z_3)-\beta_2(t_3+t_4+x_4+z_4)-\beta_3(t_1+t_3+x_1+z_1)-\beta_4(t_2+t_4+x_2+z_2)}\nonumber\\
&\quad\times (x_1x_2x_3x_4z_1z_2z_3z_4)^{r-1}(t_1t_2t_3t_4)^{r^2-1}P(1-t_1-t_2-x_3-z_3)\\
&\quad\times P(1-t_3-t_4-x_4-z_4)P(1-t_1-t_3-x_1-z_1)P(1-t_2-t_4-x_2-z_2)\nonumber\\
&\quad\times dx_1dx_2dx_3dx_4dz_1dz_2dz_3dz_4dt_1dt_2dt_3dt_4+O\big(T(\log T)^{4(r+1)^2-1}\big),\nonumber
\end{align}
where
\begin{align*}
U(\underline{x},\underline{z})&=\frac{y^{-\alpha_1(x_1+x_2)-\alpha_2(z_1+z_2)-\alpha_3(x_3+x_4)-\alpha_4(z_3+z_4)}}{(\alpha_1+\alpha_3)(\alpha_1+\alpha_4)(\alpha_2+\alpha_3)(\alpha_2+\alpha_4)}
\\
&\qquad\qquad-\frac{T^{-(\alpha_1+\alpha_3)}y^{\alpha_3(x_1+x_2)-\alpha_2(z_1+z_2)+\alpha_1(x_3+x_4)-\alpha_4(z_3+z_4)}}{(\alpha_1+\alpha_3)(-\alpha_3+\alpha_4)(\alpha_2-\alpha_1)(\alpha_2+\alpha_4)}\\
&\qquad\qquad-\frac{T^{-(\alpha_1+\alpha_4)}y^{\alpha_4(x_1+x_2)-\alpha_2(z_1+z_2)-\alpha_3(x_3+x_4)+\alpha_1(z_3+z_4)}}{(-\alpha_4+\alpha_3)(\alpha_1+\alpha_4)(\alpha_2+\alpha_3)(\alpha_2-\alpha_1)}\\
&\qquad\qquad
-\frac{T^{-(\alpha_2+\alpha_3)}y^{-\alpha_1(x_1+x_2)+\alpha_3(z_1+z_2)+\alpha_2(x_3+x_4)-\alpha_4(z_3+z_4)}}{(\alpha_1-\alpha_2)(\alpha_1+\alpha_4)(\alpha_2+\alpha_3)(-\alpha_3+\alpha_4)}\\
&\qquad\qquad-\frac{T^{-(\alpha_2+\alpha_4)}y^{-\alpha_1(x_1+x_2)+\alpha_4(z_1+z_2)-\alpha_3(x_3+x_4)+\alpha_2(z_3+z_4)}}{(\alpha_1+\alpha_3)(\alpha_1-\alpha_2)(-\alpha_4+\alpha_3)(\alpha_2+\alpha_4)}\\
&\qquad\qquad+\frac{T^{-(\alpha_1+\alpha_2+\alpha_3+\alpha_4)}y^{\alpha_3(x_1+x_2)+\alpha_4(z_1+z_2)+\alpha_1(x_3+x_4)+\alpha_2(z_3+z_4)}}{(\alpha_1+\alpha_3)(\alpha_3+\alpha_2)(\alpha_4+\alpha_1)(\alpha_2+\alpha_4)}.
\end{align*}

We proceed like in \cite[p.2317--2319]{BH}. Using the identity
\begin{eqnarray*}
\frac{1}{(\alpha_1+\alpha_3)(\alpha_1+\alpha_4)(\alpha_2+\alpha_3)(\alpha_2+\alpha_4)}&=&\frac{1}{(\alpha_1-\alpha_2)(\alpha_3-\alpha_4)(\alpha_1+\alpha_3)(\alpha_2+\alpha_4)}\\
&&\quad-\frac{1}{(\alpha_1-\alpha_2)(\alpha_3-\alpha_4)(\alpha_1+\alpha_4)(\alpha_2+\alpha_3)}
\end{eqnarray*}
we may write the first and the last terms as
\begin{eqnarray*}
\frac{y^{-\alpha_1(x_1+x_2)-\alpha_2(z_1+z_2)-\alpha_3(x_3+x_4)-\alpha_4(z_3+z_4)}}{(\alpha_1+\alpha_3)(\alpha_1+\alpha_4)(\alpha_2+\alpha_3)(\alpha_2+\alpha_4)}&=&\frac{y^{-\alpha_1(x_1+x_2)-\alpha_2(z_1+z_2)-\alpha_3(x_3+x_4)-\alpha_4(z_3+z_4)}}{(\alpha_1-\alpha_2)(\alpha_3-\alpha_4)(\alpha_1+\alpha_3)(\alpha_2+\alpha_4)}\\
&&\quad-\frac{y^{-\alpha_1(x_1+x_2)-\alpha_2(z_1+z_2)-\alpha_3(x_3+x_4)-\alpha_4(z_3+z_4)}}{(\alpha_1-\alpha_2)(\alpha_3-\alpha_4)(\alpha_1+\alpha_4)(\alpha_2+\alpha_3)}
\end{eqnarray*}
and
\begin{align}\label{swap1}
&\frac{T^{-(\alpha_1+\alpha_2+\alpha_3+\alpha_4)}y^{\alpha_3(x_1+x_2)+\alpha_4(z_1+z_2)+\alpha_1(x_3+x_4)+\alpha_2(z_3+z_4)}}{(\alpha_1+\alpha_3)(\alpha_3+\alpha_2)(\alpha_4+\alpha_1)(\alpha_2+\alpha_4)}\nonumber\\
&\qquad\quad=\frac{T^{-(\alpha_1+\alpha_2+\alpha_3+\alpha_4)}y^{\alpha_3(x_1+x_2)+\alpha_4(z_1+z_2)+\alpha_1(x_3+x_4)+\alpha_2(z_3+z_4)}}{(\alpha_1-\alpha_2)(\alpha_3-\alpha_4)(\alpha_1+\alpha_3)(\alpha_2+\alpha_4)}\nonumber\\
&\qquad\qquad\qquad-\frac{T^{-(\alpha_1+\alpha_2+\alpha_3+\alpha_4)}y^{\alpha_3(x_1+x_2)+\alpha_4(z_1+z_2)+\alpha_1(x_3+x_4)+\alpha_2(z_3+z_4)}}{(\alpha_1-\alpha_2)(\alpha_3-\alpha_4)(\alpha_1+\alpha_4)(\alpha_2+\alpha_3)}.
\end{align}
Notice from Lemma \ref{600} that we can change the roles $x_j\longleftrightarrow z_j$ in any term of $U(\underline{x},\underline{z})$ without affecting the value of $I(\underline{\alpha},\underline{\beta})$ in \eqref{Uformula}. So by applying all these changes to the last term on the right hand of \eqref{swap1}, we may replace $U(\underline{x},\underline{z})$ with
\begin{align}\label{combin}
&\frac{y^{-\alpha_1(x_1+x_2)-\alpha_2(z_1+z_2)-\alpha_3(x_3+x_4)-\alpha_4(z_3+z_4)}}{(\alpha_1-\alpha_2)(\alpha_3-\alpha_4)}\bigg(\frac{1-(Ty^{-\sum x_j})^{-(\alpha_1+\alpha_3)}}{\alpha_1+\alpha_3}\bigg)\nonumber\\
&\qquad\qquad\times\bigg(\frac{1-(Ty^{-\sum z_j})^{-(\alpha_2+\alpha_4)}}{\alpha_2+\alpha_4}\bigg)\nonumber\\
&\ -\frac{y^{-\alpha_1(x_1+x_2)-\alpha_2(z_1+z_2)-\alpha_3(x_3+x_4)-\alpha_4(z_3+z_4)}}{(\alpha_1-\alpha_2)(\alpha_3-\alpha_4)}\bigg(\frac{1-(Ty^{-(x_1+x_2+z_3+z_4)})^{-(\alpha_1+\alpha_4)}}{\alpha_1+\alpha_4}\bigg)\\
&\qquad\qquad\times\bigg(\frac{1-(Ty^{-(z_1+z_2+x_3+x_4)})^{-(\alpha_2+\alpha_3)}}{\alpha_2+\alpha_3}\bigg).\nonumber
\end{align}
From the integral formula
\begin{equation}\label{int}
\frac{1-y^{-(\alpha+\beta)}}{\alpha+\beta}=(\log y)\int_{0}^{1}y^{-(\alpha+\beta)v}dv
\end{equation} we obtain that
\begin{align}\label{finalformulaI}
&I(\underline{\alpha},\underline{\beta})=\frac{\hat{\omega}(0)a_{2r+2}T(\log y)^{4r^2+8r}(\log T)^2}{((r-1)!)^8((r^2-1)!)^4}\ \idotsint\limits_{\substack{0\leq v_j,t_j,x_j,z_j\leq 1\\t_1+t_2+x_3+z_3\leq1\\t_3+t_4+x_4+z_4\leq1\\t_1+t_3+x_1+z_1\leq 1\\t_2+t_4+x_2+z_2\leq 1 }}\frac{U_1(\underline{x},\underline{z},v_1,v_2)-U_2(\underline{x},\underline{z},v_1,v_2)}{(\alpha_1-\alpha_2)(\alpha_3-\alpha_4)}\nonumber\\
&\qquad\times y^{-\beta_1(t_1+t_2+x_3+z_3)-\beta_2(t_3+t_4+x_4+z_4)-\beta_3(t_1+t_3+x_1+z_1)-\beta_4(t_2+t_4+x_2+z_2)}\nonumber\\
&\qquad\times (x_1x_2x_3x_4z_1z_2z_3z_4)^{r-1}(t_1t_2t_3t_4)^{r^2-1}P(1-t_1-t_2-x_3-z_3)\\
&\qquad\times P(1-t_3-t_4-x_4-z_4)P(1-t_1-t_3-x_1-z_1)P(1-t_2-t_4-x_2-z_2)\nonumber\\
&\qquad\times dx_1dx_2dx_3dx_4dz_1dz_2dz_3dz_4dt_1dt_2dt_3dt_4dv_1dv_2+O\big(T(\log T)^{4(r+1)^2-1}\big),\nonumber
\end{align}
where
\begin{align*}
U_1&=y^{-\alpha_1(x_1+x_2)-\alpha_2(z_1+z_2)-\alpha_3(x_3+x_4)-\alpha_4(z_3+z_4)}(Ty^{-\sum x_j})^{-(\alpha_1+\alpha_3)v_1}(Ty^{-\sum z_j})^{-(\alpha_2+\alpha_4)v_2}\\
&\qquad\qquad\times\big(1-\vartheta\sum x_j\big)\big(1-\vartheta\sum z_j\big)
\end{align*}
and
\begin{align*}
&U_2= y^{-\alpha_1(x_1+x_2)-\alpha_2(z_1+z_2)-\alpha_3(x_3+x_4)-\alpha_4(z_3+z_4)}(Ty^{-(x_1+x_2+z_3+z_4)})^{-(\alpha_1+\alpha_4)v_1}\\
&\ \quad\times(Ty^{-(z_1+z_2+x_3+x_4)})^{-(\alpha_2+\alpha_3)v_2}\big(1-\vartheta(x_1+x_2+z_3+z_4)\big)\big(1-\vartheta(z_1+z_2+x_3+x_4)\big).
\end{align*}

Again note that $I(\underline{\alpha},\underline{\beta})$ is unchanged if we swap any of the pairs of variables $x_j\longleftrightarrow z_j$ for any $1\leq j\leq 4$ or $v_1\longleftrightarrow v_2$ in $U_1$ or $U_2$. Hence we can replace the term $U_1-U_2$ in the integrand with
\begin{eqnarray*}
&&\frac{1}{2}\Big(U_1(\underline{x},\underline{z},v_1,v_2)-U_2(z_1,z_2,x_3,x_4,x_1,x_2,z_3,z_4,v_2,v_1)\\
&&\qquad\qquad-U_2(x_1,x_2,z_3,z_4,z_1,z_2,x_3,x_4,v_1,v_2)+U_1(\underline{z},\underline{x},v_2,v_1)\Big),
\end{eqnarray*}
which is
\begin{align*}
&\frac{\big(1-\vartheta\sum x_j\big)\big(1-\vartheta\sum z_j\big)}{2}\nonumber\\
&\quad\times\bigg(y^{-\alpha_1(x_1+x_2)-\alpha_2(z_1+z_2)-\alpha_3(x_3+x_4)-\alpha_4(z_3+z_4)}(Ty^{-\sum x_j})^{-(\alpha_1+\alpha_3)v_1}(Ty^{-\sum z_j})^{-(\alpha_2+\alpha_4)v_2}\nonumber\\
&\qquad\ \  -y^{-\alpha_1(z_1+z_2)-\alpha_2(x_1+x_2)-\alpha_3(x_3+x_4)-\alpha_4(z_3+z_4)}(Ty^{-\sum z_j})^{-(\alpha_1+\alpha_4)v_2}(Ty^{-\sum x_j})^{-(\alpha_2+\alpha_3)v_1}\nonumber\\
&\qquad\ \ - y^{-\alpha_1(x_1+x_2)-\alpha_2(z_1+z_2)-\alpha_3(z_3+z_4)-\alpha_4(x_3+x_4)}(Ty^{-\sum x_j})^{-(\alpha_1+\alpha_4)v_1}(Ty^{-\sum z_j})^{-(\alpha_2+\alpha_3)v_2}\nonumber\\
&\qquad\ \ +y^{-\alpha_1(z_1+z_2)-\alpha_2(x_1+x_2)-\alpha_3(z_3+z_4)-\alpha_4(x_3+x_4)}(Ty^{-\sum z_j})^{-(\alpha_1+\alpha_3)v_2}(Ty^{-\sum x_j})^{-(\alpha_2+\alpha_4)v_1}\bigg)\nonumber\\
&\ =\frac{\big(1-\vartheta\sum x_j\big)\big(1-\vartheta\sum z_j\big)}{2}y^{-\alpha_1(x_1+x_2)-\alpha_2(z_1+z_2)-\alpha_3(x_3+x_4)-\alpha_4(z_3+z_4)}(Ty^{-\sum x_j})^{-(\alpha_1+\alpha_3)v_1}\nonumber\\
&\qquad\times(Ty^{-\sum z_j})^{-(\alpha_2+\alpha_4)v_2}\bigg(1-\Big(T^{v_1-v_2}y^{x_1+x_2-z_1-z_2-v_1\sum x_j+v_2\sum z_j}\Big)^{\alpha_1-\alpha_2}\bigg)\\
&\qquad\times\bigg(1-\Big(T^{v_1-v_2}y^{x_3+x_4-z_3-z_4-v_1\sum x_j+v_2\sum z_j}\Big)^{\alpha_3-\alpha_4}\bigg).\nonumber
\end{align*}
Using \eqref{int} again in \eqref{finalformulaI} and simplifying we obtain the theorem.

\section{Large gaps between zeros of $\zeta(s)$ and proof of Theorem \ref{gapthm}} \label{sectionW}

With our limited knowledge, there is only one known method to detect large gaps between consecutive zeros of the Riemann zeta-function unconditionally, and this was discovered by the second named author. A key ingredient is the use of some Wirtinger type inequalities. We shall start with the basic idea and then discuss the two known approaches.

\subsection{The Wirtinger inequality and large gaps between zeros}\label{basicidea}

Let $a,b\in\mathbb{R}$ and $f\in C^1$ with $f(a)=f(b)=0$. Then the Wirtinger inequality states that
\begin{equation}\label{firstWineq}
\int_{a}^{b}|f(t)|^2dt\leq \Big(\frac{b-a}{\pi}\Big)^2\int_{a}^{b}|f'(t)|^2dt .
\end{equation}

Let $Z(t)$ be the Hardy $Z$-function. In particular, $|Z(t_n)|=|\zeta(1/2+it_n)|=0$ for all $n\in\mathbb{N}$. Consider the hypothesis  $
\Lambda\leq \kappa$.
Then for $t_n\in[T,2T]$ we have
\begin{equation*}
t_{n+1}-t_n\leq \big(1+o(1)\big)\frac{2\pi \kappa}{\log T},
\end{equation*}
as $T\rightarrow\infty$. In view of \eqref{firstWineq} we then get
\begin{align*}
\int_{t_n}^{t_{n+1}}|Z(t)|^2dt&\leq\big(1+o(1)\big)\frac{4\kappa^2}{(\log T)^2}\int_{t_n}^{t_{n+1}}|Z'(t)|^2dt.
\end{align*}
Let $t_n,t_m\in[T,2T]$ be the zeros closest to $T$ and $2T$, respectively. Summing up the above inequalities we obtain that
\begin{align*}
\int_{t_n}^{t_{m}}|Z(t)|^2dt&\leq\big(1+o(1)\big)\frac{4\kappa^2}{(\log T)^2}\int_{t_n}^{t_{m}}|Z'(t)|^2dt.
\end{align*}
As the integrals above are both $\gg T$, and by our hypothesis we have $t_n-T\ll 1$ and $t_m-2T\ll 1$, it follows that 
\begin{align*}
\int_{T}^{2T}|Z(t)|^2dt&\leq\big(1+o(1)\big)\frac{4\kappa^2}{(\log T)^2}\int_{T}^{2T}|Z'(t)|^2dt.
\end{align*}
Hence if
\[
\int_{T}^{2T}|Z(t)|^2dt>\frac{4\kappa^2}{(\log T)^2}\int_{T}^{2T}|Z'(t)|^2dt,
\]
then we must have $\Lambda>\kappa$.

It is known that 
\[
\int_{T}^{2T}|Z^{(k)}(t)|^2dt\sim \frac{T(\log T)^{2k+1}}{4^k(2k+1)}
\]
as $T\rightarrow\infty$ (see, for instance, \cite[Theorem 3]{H3}). We therefore get $\Lambda\geq\sqrt{3}=1.73\ldots$.

\subsection{Hall's first approach}\label{Hall1st}

The key ingredient in Hall's first attempt is a Wirtinger type inequality due to Beesack \cite{Bee}, which generally states that 
\begin{equation}\label{secWineq}
\int_{a}^{b}|f(t)|^{2k}dt\leq \frac{1}{2k-1}\Big(k\sin\frac{\pi}{2k}\Big)^{2k}\Big(\frac{b-a}{\pi}\Big)^{2k}\int_{a}^{b}|f'(t)|^{2k}dt,
\end{equation}
where $f\in C^1$ with $f(a)=f(b)=0$. The Wirtinger inequality \eqref{firstWineq} is the special case $k=1$ of \eqref{secWineq}. Using \eqref{secWineq} with $k=2$, together with  the idea described above and the fourth moments of the Riemann zeta-function and its derivative, Hall \cite{H3} showed that $\Lambda\geq \sqrt[4]{105/4}=2.26\ldots$. With a more general inequality when $k=2$ (see Theorem \ref{Wi} below), the result was later improved to $\Lambda\geq \sqrt{11/2}=2.34\ldots$ \cite{H2}.

\subsection{Hall's second approach}

Hall's later approach just uses the standard Wirtinger inequality \eqref{firstWineq} rather than the $L^4$ form as described in Section \ref{Hall1st}. The key idea is to apply \eqref{firstWineq} to 
\begin{equation}\label{fformula}
f(t)=Z(t)Z\Big(t+\frac{2\pi\kappa}{\log T}\Big)
\end{equation}
and observe that if $f(t)$ does not vanish in the interior of a closed interval of length $2\kappa\pi/\log T$, then $Z(t)$ does not vanish in the interior of a closed interval of length $4\kappa\pi/\log T$. So, arguing as in Section \ref{basicidea}, if
\[
\int_{T}^{2T}|f(t)|^2dt>\frac{4\kappa^2}{(\log T)^2}\int_{T}^{2T}|f'(t)|^2dt,
\]
then we must have $\Lambda>2\kappa$. The drawback of the employment of the standard Wirtinger inequality is hence compensated by the doubling of the shift, and Hall \cite{H} obtained that $\Lambda>2.63$.

The later improvement of Bui and Milinovich \cite{BM}, and also of Bredberg \cite{B1}, comes from introducing an amplifier to $f$ in \eqref{fformula}. Instead of \eqref{fformula}, Bui and Milinovich \cite{BM} applied the  Wirtinger inequality \eqref{firstWineq} to
\[
\widetilde{f}(t)=Z(t)Z\Big(t+\frac{2\pi\kappa}{\log T}\Big)A(\tfrac12+it),
\]
where $A$ is an amplifier of length $T^\vartheta$ with $\vartheta<1/4$. Their extra ingredient is the mean value of the fourth power of $\zeta(s)$ times the square of an amplifier. Using our Theorem  \ref{mainthm} one can instead consider 
\[
\widetilde{f}(t)=Z(t)Z\Big(t+\frac{2\pi\kappa}{\log T}\Big)A(\tfrac12+it)A\Big(\tfrac12+it+\frac{2\pi i\kappa}{\log T}\Big),
\]
where $A$ is an amplifier of length $T^\vartheta$ with $\vartheta<1/8$. This will also lead to an improvement to Hall's result $\Lambda>2.63$.

\subsection{Hall's Wirtinger type inequality}

In this paper we shall employ a Wirtinger type inequality essentially due to Hall. For the application to large gaps between zeros of the Riemann zeta-function, we shall use the amplified fourth moment of $\zeta(s)$ in Theorem \ref{mainthm} rather than just the fourth moment of $\zeta(s)$ as in \cite{H2}.

\begin{theorem}\label{Wi}
Let $a,b\in\mathbb{R}$ and let $f$ be a differentiable complex-valued function with $f(a)=f(b)=0$. Then for any $v\geq0$ we have
\begin{equation*}
3\lambda_0(v)\int_{a}^{b}|f(t)|^4dt\leq \Big(\frac{b-a}{\pi}\Big)^4\int_{a}^{b}|f'(t)|^4dt +6v\Big(\frac{b-a}{\pi}\Big)^2\int_{a}^{b}|f(t)f'(t)|^2dt,
\end{equation*}
where
\[
\lambda_0(v)=\frac{1+4v+\sqrt{1+8v}}{8}.
\]
\end{theorem}
\begin{proof}
This theorem was proved by Hall \cite[Theorem 2]{H2} for continuously differentiable real-valued functions. For completeness we will show that it also holds for complex-valued functions. 

For $t\in\mathbb{R}$, let $f(t)=u(t)+iv(t)$. Then $g(t)=|f(t)|$ is continuously differentiable except at the real zeros of $f$, which are isolated. Hence by applying Theorem 2 in \cite{H2} to $g$, we obtain that
\begin{equation}\label{eq3}
3\lambda_0(v)\int_{a}^{b}|f(t)|^4dt\leq \Big(\frac{b-a}{\pi}\Big)^4\int_{a}^{b}g'(t)^4dt +6v\Big(\frac{b-a}{\pi}\Big)^2\int_{a}^{b}|f(t)|^2g'(t)^2dt.
\end{equation}

On one hand we have
\[
g(t)^2=u(t)^2+v(t)^2.
\]
Differentiating gives
\[
2|f(t)|g'(t)=2u(t)u'(t)+2v(t)v'(t),
\]
and hence
\begin{equation}\label{eq1}
|f(t)|^2g'(t)^2=u(t)^2u'(t)^2+v(t)^2v'(t)^2+2u(t)v(t)u'(t)v'(t).
\end{equation}
On the other hand, $f'(t)=u'(t)+iv'(t)$, and we have
\begin{align}\label{eq2}
|f(t)|^2|f'(t)|^2&=(u(t)^2+v(t)^2)(u'(t)^2+v'(t)^2)\nonumber\\
&=u(t)^2u'(t)^2+v(t)^2v'(t)^2+u(t)^2v'(t)^2+u'(t)^2v(t)^2.
\end{align}
Comparing \eqref{eq1} and \eqref{eq2} we get
\[
|f(t)|^2g'(t)^2\leq |f(t)|^2|f'(t)|^2.
\]
Thus, $|g'(t)|\leq|f'(t)|$, and in view of \eqref{eq3} we obtain the theorem.
\end{proof}

\subsection{Reduction to mean value theorems}

We shall consider
\[
f(t):=f(t,u,A)=e^{iu(\log T)t}\zeta(\tfrac{1}{2}+it)A(\tfrac{1}{2}+it),
\]
where $u\in\mathbb{R}$ to be chosen later. Denote all the zeros of the function $f$ in the interval $[T,2T]$ by $\widetilde{t}_1\leq \widetilde{t}_2\leq\ldots\leq \widetilde{t}_N$ . Suppose, for the sake of contradiction, that 
\[
\Lambda\leq\kappa.
\]
Then we have
\[
\widetilde{t}_{n+1}-\widetilde{t}_n\leq \big(1+o(1)\big)\frac{2\pi \kappa}{\log T},
\]
as $T\rightarrow\infty$ for $1\leq n\leq N-1$. It follows from Theorem \ref{Wi} that
\begin{align*}
3\lambda_0(v)\int_{\widetilde{t}_n}^{\widetilde{t}_{n+1}}|f(t)|^4dt&\leq\big(1+o(1)\big)\frac{16\kappa^4}{(\log T)^4}\int_{\widetilde{t}_n}^{\widetilde{t}_{n+1}}|f'(t)|^4dt\\
&\qquad\qquad+\big(6v+o(1)\big)\frac{4\kappa^2}{(\log T)^2}\int_{\widetilde{t}_n}^{\widetilde{t}_{n+1}}|f(t)f'(t)|^2dt
\end{align*}
for $1\leq n\leq N-1$. Summing up as in Section \ref{basicidea} we have
\begin{align*}
3\lambda_0(v)\int_{T}^{2T}|f(t)|^4dt&\leq\big(1+o(1)\big)\frac{16\kappa^4}{(\log T)^4}\int_{T}^{2T}|f'(t)|^4dt\\
&\qquad\qquad+\big(6v+o(1)\big)\frac{4\kappa^2}{(\log T)^2}\int_{T}^{2T}|f(t)f'(t)|^2dt.
\end{align*}
Hence provided that 
\begin{align}\label{keyineq}
3\lambda_0(v)\int_{T}^{2T}|f(t)|^4dt>\frac{16\kappa^4}{(\log T)^4}\int_{T}^{2T}|f'(t)|^4dt+\frac{24v\kappa^2}{(\log T)^2}\int_{T}^{2T}|f(t)f'(t)|^2dt,
\end{align}
then we obtain that $\Lambda>\kappa$.

\subsection{Proof of Theorem \ref{gapthm}}

We have
\[
|f(t)|=\big|\zeta(\tfrac{1}{2}+it)A(\tfrac{1}{2}+it)\big|,
\]
and hence
\[
\int_{T}^{2T}|f(t)|^4dt=\frac{a_{2r+2}c(\underline{0},\underline{0})}{2((r-1)!)^8((r^2-1)!)^4}T(\log y)^{4r^2+8r}(\log T)^{4}+O\big(T(\log T)^{4(r+1)^2-1}\big),
\]
by Theorem \ref{mainthm}.

With $f'(t)$ we have
\begin{align*}
\frac{f'(t)}{ie^{iu(\log T)t}}&=u(\log T)\zeta\big(\tfrac{1}{2}+it\big)A\big(\tfrac{1}{2}+it\big) +\Big(\frac{d}{d\alpha}+\frac{d}{d\beta}\Big)\zeta\big(\tfrac{1}{2}+\alpha+it\big)A\big(\tfrac{1}{2}+\beta+it\big)\bigg|_{\alpha=\beta=0}\\
&=(\log T)Q\bigg(\frac{1}{\log T}\Big(\frac{d}{d\alpha}+\frac{d}{d\beta}\Big)\bigg)\zeta\big(\tfrac{1}{2}+\alpha+it\big)A\big(\tfrac{1}{2}+\beta+it\big)\bigg|_{\alpha=\beta=0},
\end{align*}
where
\begin{equation}\label{Qformula}
Q(x)=u+x.
\end{equation}
Therefore,
\begin{align}\label{abc}
&\int_{-\infty}^{\infty}|f'(t)|^4w\Big(\frac tT\Big)dt=(\log T)^4Q\bigg(\frac{1}{\log T}\Big(\frac{d}{d\alpha_1}+\frac{d}{d\beta_1}\Big)\bigg)Q\bigg(\frac{1}{\log T}\Big(\frac{d}{d\alpha_2}+\frac{d}{d\beta_2}\Big)\bigg)\nonumber\\
&\qquad\times Q\bigg(\frac{1}{\log T}\Big(\frac{d}{d\alpha_3}+\frac{d}{d\beta_3}\Big)\bigg)Q\bigg(\frac{1}{\log T}\Big(\frac{d}{d\alpha_4}+\frac{d}{d\beta_4}\Big)\bigg)I(\underline{\alpha},\underline{\beta})\bigg|_{\substack{\underline{\alpha}=\underline{\beta}=\underline{0}}}.
\end{align}

Since $I(\underline{\alpha},\underline{\beta})$ is holomorphic with respect to $\alpha_j, \beta_j$ small, the derivatives in \eqref{abc} can be obtained as integrals of radii $\asymp (\log T)^{-1}$ around the points $\alpha_j=\beta_j=0$, $1\leq j\leq 4$, using Cauchy's integral formula.  Since the error term holds uniformly on these contours, the error term that holds for $I(\underline{\alpha},\underline{\beta})$ also holds for its derivatives. Thus
\begin{align*}
\frac{1}{(\log T)^4}\int_{T}^{2T}|f'(t)|^4dt&=\frac{a_{2r+2}c_1(u)}{2((r-1)!)^8((r^2-1)!)^4}T(\log y)^{4r^2+8r}(\log T)^{4}\\
&\qquad\qquad+O\big(T(\log T)^{4(r+1)^2-1}\big),
\end{align*}
where
\begin{align*}
c_1(u)&=Q\bigg(\frac{1}{\log T}\Big(\frac{d}{d\alpha_1}+\frac{d}{d\beta_1}\Big)\bigg)Q\bigg(\frac{1}{\log T}\Big(\frac{d}{d\alpha_2}+\frac{d}{d\beta_2}\Big)\bigg)\nonumber\\
&\qquad\qquad\times Q\bigg(\frac{1}{\log T}\Big(\frac{d}{d\alpha_3}+\frac{d}{d\beta_3}\Big)\bigg)Q\bigg(\frac{1}{\log T}\Big(\frac{d}{d\alpha_4}+\frac{d}{d\beta_4}\Big)\bigg)c(\underline{\alpha},\underline{\beta})\bigg|_{\substack{\underline{\alpha}=\underline{\beta}=\underline{0}}}.
\end{align*}

Note that
\begin{align*}
Q\bigg(\frac{1}{\log T}\Big(\frac{d}{d\alpha}+\frac{d}{d\beta}\Big)\bigg)X^{\alpha}Y^\beta&=Q\Big(\frac{\log X+\log Y}{\log T}\Big)X^{\alpha}Y^\beta\\
&=\Big(u+\frac{\log X+\log Y}{\log T}\Big)X^{\alpha}Y^\beta,
\end{align*}
by \eqref{Qformula}. So it follows from the expression of $c(\underline{\alpha},\underline{\beta})$ in Theorem \ref{mainthm} that
\begin{align*}
	&c_1(u)= \idotsint\limits_{\substack{0\leq v_j,t_j,x_j,z_j\leq 1\\t_1+t_2+x_3+z_3\leq1\\t_3+t_4+x_4+z_4\leq1\\t_1+t_3+x_1+z_1\leq 1\\t_2+t_4+x_2+z_2\leq 1 }}\big(1-\vartheta\sum x_j\big)\big(1-\vartheta\sum z_j\big)\\
	&\quad\times\Big(v_1-v_2+\vartheta\big(x_1+x_2-z_1-z_2-v_1\sum x_j+v_2\sum z_j\big)\Big)\\
	&\quad\times\Big(v_1-v_2+\vartheta\big(x_3+x_4-z_3-z_4-v_1\sum x_j+v_2\sum z_j\big)\Big)\\
	&\quad\times \Big(u-\vartheta(t_1+t_2+x_1+x_2+x_3+z_3)-v_1(1-\vartheta\sum x_j)\\
	&\qquad\qquad\qquad+v_3\big(v_1-v_2+\vartheta(x_1+x_2-z_1-z_2-v_1\sum x_j+v_2\sum z_j)\big)\Big)\\
	 &\quad\times \Big(u-\vartheta(t_3+t_4+x_4+z_1+z_2+z_4)-v_2(1-\vartheta\sum z_j)\\
	&\qquad\qquad\qquad-v_3\big(v_1-v_2+\vartheta(x_1+x_2-z_1-z_2-v_1\sum x_j+v_2\sum z_j)\big) \Big)\\
	&\quad\times \Big(u-\vartheta(t_1+t_3+x_1+x_3+x_4+z_1)-v_1(1-\vartheta\sum x_j)\\
		&\qquad\qquad\qquad+v_4\big(v_1-v_2+\vartheta\big(x_3+x_4-z_3-z_4-v_1\sum x_j+v_2\sum z_j)\big)\Big)\\
    &\quad\times \Big(u-\vartheta(t_2+t_4+x_2+z_2+z_3+z_4)-v_2(1-\vartheta\sum z_j)\\
    &\qquad\qquad\qquad-v_4\big(v_1-v_2+\vartheta\big(x_3+x_4-z_3-z_4-v_1\sum x_j+v_2\sum z_j)\big)\Big)\\
	&\quad\times (x_1x_2x_3x_4z_1z_2z_3z_4)^{r-1}(t_1t_2t_3t_4)^{r^2-1}P(1-t_1-t_2-x_3-z_3)\\
	&\quad\times P(1-t_3-t_4-x_4-z_4)P(1-t_1-t_3-x_1-z_1)P(1-t_2-t_4-x_2-z_2)\\
	&\quad\times dx_1dx_2dx_3dx_4dz_1dz_2dz_3dz_4dt_1dt_2dt_3dt_4dv_1dv_2dv_3dv_4,
\end{align*}

Similarly,
\begin{align*}
\frac{1}{(\log T)^2}\int_{T}^{2T}|f(t)f'(t)|^2dt&=\frac{a_{2r+2}c_2(u)}{2((r-1)!)^8((r^2-1)!)^4}T(\log y)^{4r^2+8r}(\log T)^{4}\\
&\qquad\qquad+O\big(T(\log T)^{4(r+1)^2-1}\big),
\end{align*}
where
\begin{align*}
c_2(u)&=Q\bigg(\frac{1}{\log T}\Big(\frac{d}{d\alpha_1}+\frac{d}{d\beta_1}\Big)\bigg)Q\bigg(\frac{1}{\log T}\Big(\frac{d}{d\alpha_3}+\frac{d}{d\beta_3}\Big)\bigg)c(\underline{\alpha},\underline{\beta})\bigg|_{\substack{\underline{\alpha}=\underline{\beta}=\underline{0}}}.
\end{align*}
Hence
\begin{align*}
	&c_2(u)= \idotsint\limits_{\substack{0\leq v_j,t_j,x_j,z_j\leq 1\\t_1+t_2+x_3+z_3\leq1\\t_3+t_4+x_4+z_4\leq1\\t_1+t_3+x_1+z_1\leq 1\\t_2+t_4+x_2+z_2\leq 1 }}\big(1-\vartheta\sum x_j\big)\big(1-\vartheta\sum z_j\big)\\
	&\quad\times\Big(v_1-v_2+\vartheta\big(x_1+x_2-z_1-z_2-v_1\sum x_j+v_2\sum z_j\big)\Big)\\
	&\quad\times\Big(v_1-v_2+\vartheta\big(x_3+x_4-z_3-z_4-v_1\sum x_j+v_2\sum z_j\big)\Big)\\
	&\quad\times \Big(u-\vartheta(t_1+t_2+x_1+x_2+x_3+z_3)-v_1(1-\vartheta\sum x_j)\\
	&\qquad\qquad\qquad+v_3\big(v_1-v_2+\vartheta(x_1+x_2-z_1-z_2-v_1\sum x_j+v_2\sum z_j)\big)\Big)\\
	&\quad\times \Big(u-\vartheta(t_1+t_3+x_1+x_3+x_4+z_1)-v_1(1-\vartheta\sum x_j)\\
		&\qquad\qquad\qquad+v_4\big(v_1-v_2+\vartheta\big(x_3+x_4-z_3-z_4-v_1\sum x_j+v_2\sum z_j)\big)\Big)\\
	&\quad\times (x_1x_2x_3x_4z_1z_2z_3z_4)^{r-1}(t_1t_2t_3t_4)^{r^2-1}P(1-t_1-t_2-x_3-z_3)\\
	&\quad\times P(1-t_3-t_4-x_4-z_4)P(1-t_1-t_3-x_1-z_1)P(1-t_2-t_4-x_2-z_2)\\
	&\quad\times dx_1dx_2dx_3dx_4dz_1dz_2dz_3dz_4dt_1dt_2dt_3dt_4dv_1dv_2dv_3dv_4,
\end{align*}

In view of \eqref{keyineq} we obtain that $\Lambda>\kappa$ if
\begin{align}\label{keyineq2}
3\lambda_0(v)c(\underline{0},\underline{0})>16\kappa^4c_1(u)+24v\kappa^2c_2(u).
\end{align}
Notice that \eqref{keyineq2} is equivalent to
\[
\kappa^2<\frac{\sqrt{144v^2c_2(u)^2+48\lambda_0(v)c(\underline{0},\underline{0})c_1(u)}-12vc_2(u)}{16c_1(u)},
\]
and hence
\begin{align}\label{keyineq3}
\Lambda^2\geq\frac{\sqrt{144v^2c_2(u)^2+48\lambda_0(v)c(\underline{0},\underline{0})c_1(u)}-12vc_2(u)}{16c_1(u)}.
\end{align}

In the case $\vartheta=0$, $r=1$, $P(x)=1$ and $u=1/2$, the inequality \eqref{keyineq3} reduces to
\[
\Lambda^2\geq \sqrt{49v^2+105\lambda_0(v)}-7v.
\] 
The right hand side above attains its maximum when $v=22/49$, and we recover Hall's result $\Lambda\geq\sqrt{11/2}$ in \cite[Theorem 1]{H2}.

With an amplifier of length $T^\vartheta$, where  $\vartheta=0.1249$, the choice $r=1$, $P(x)=1$, $u=0.6$ and $v=0.4$ gives $\Lambda>2.64$ as claimed.\\

\noindent{\sc Acknowledgments}. MSJ was supported by the Dame Kathleen Ollerenshaw (DKO) Research Internship when the work commenced.

\end{document}